\documentclass[12pt]{article}
\makeatletter
\@addtoreset{equation}{section}
\makeatother
\usepackage{bm}
\usepackage{amsmath}
\usepackage{amsfonts}
\usepackage{latexsym}
\usepackage{amssymb}
\usepackage{mathrsfs}
\usepackage{young}

\usepackage{ntheorem}

\newtheorem{theorem}{Theorem}[section]
\newtheorem{proposition}[theorem]{Proposition}
\newtheorem{corollary}[theorem]{Corollary}
\newtheorem{lemma}[theorem]{Lemma}

{\theoremheaderfont{\normalfont}
\newtheorem{definition}[theorem]{Definition}

\newtheorem{remark}[theorem]{Remark}
\newtheorem{acknowledgments}{Acknowledgments}
\newtheorem{proof}{Proof}
}

\begin{document}
\title{On quasiinvariants of $S_n$ of hook shape}
\author{T. Tsuchida}
\date{}

\maketitle

\begin{abstract}
{\rm 
Chalykh, Veselov and Feigin introduced the notions of quasiinvariants for Coxeter groups,
which is a generalization of invariants. 
In \cite{bandlow}, Bandlow and Musiker showed that for the symmetric group $S_n$ of order $n$,
the space of quasiinvariants has
a decomposition indexed by standard tableaux. 
They gave a description of basis for the components indexed by standard tableaux of shape $(n-1,1)$. 
In this paper, we generalize their results to a 
description of basis for the components indexed by standard tableaux of arbitrary hook shape.
}
\end{abstract}
\section{Introduction}
In \cite{Chalykh} and \cite{feigin}, Chalykh, Veselov and Feigin introduced the notions of $quasi$\\
$invariants$ for Coxeter groups, which is a generalization of invariants. 
For any Coxeter group, the quasiinvariants is defined by giving a multiplicity $m$ which is a map 
from the conjugacy classes of the corresponding group to non-negative integers.

In the case of $S_n$, the multiplicity is a constant function.
Take a non-negative integer $m$.
A polynomial $P\in \mathbb{Q}[x_1,x_2,\cdots ,x_n]$ is called 
a $m$-quasiinvariants if the difference
\begin{equation*}
\bigl(1-(i,j) \bigr) P(x_1,\cdots ,x_n)
\end{equation*}
is divisible by $(x_i-x_j)^{2m+1}$ for any transposition $(i,j)\in S_n$. 

The notion of quasiinvariants was first introduced
in the study of the quantum Calogero Moser system. 
In the case of $S_n$, this system is defined by the following differential operator 
(the generalized Calogero-Moser Hamiltonian):
\begin{equation*}
L_m=\sum_{i=1}^n\frac{\partial^2}{\partial x_i^2}-2m\sum_{1\leq i < j\leq n}
\frac{1}{x_i-x_j}\bigl( \frac{\partial}{\partial x_i} - \frac{\partial}{\partial x_j} \bigr)
\end{equation*}
where $m$ is a real number.

For a Coxeter group $G$, we denote by $S^G$ the ideal generated by homogeneous invariant polynomials for $G$
and by $S^G_+$ the ideal generated by the homogeneous invariant polynomials of positive degree.
For a generic multiplicity, there exist a isomorphism from $S^G$ 
to the ring of $G$-invariant quantum integrals of the generalized Calogero-Moser Hamiltonian 
(sometimes called Harish-Chandra isomorphism).
We denote by ${\cal L}_1,{\cal L}_2,\cdots ,{\cal L}_n$ 
the operators corresponding to fundamental invariant polynomials $\sigma_1,\sigma_2, \cdots ,\sigma_n$.
the generalized Calogero-Moser Hamiltonian is a member of this ring
(see for example  \cite{feigin}, \cite{feigin2}).

In the case of non-negative integer multiplicities, Chalykh and Veselov showed 
that there is a homomorphism from the ring of quasiinvariants to the commutative ring of 
differential operators whose coefficients are rational functions (see e.g. \cite{Chalykh}).
It is shown that the restriction of such homomorphism onto $S^G$ induces Harish-Chandra isomorphism. 
Therefore, in the case of non-negative integer multiplicities there are much more quantum integrals.

Let $m$ be a non-negative multiplicity.
In \cite{feigin}, Feigin and Veselov introduced the notions of $m$-harmonics which are defined 
as the solutions of the following system:
\begin{eqnarray*}
{\cal L}_1 \psi &=&0 \\
{\cal L}_2 \psi &=&0 \\
&\vdots& \\
{\cal L}_n \psi &=&0.
\end{eqnarray*}
Feigin and Veselov also showed that the solutions of such system are polynomials.
They also showed that the space of $m$-harmonic polynomials is a subspace of 
$m$-quasiinvariants of dimension $|G|$. 
In \cite{felder}, Felder and Veselov gave a formula of the Hilbert series of 
the space of $m$-harmonic polynomials.

In \cite{etingof}, Etingof and Ginzburg proved the following:
(i) quasiinvariants of $G$ is a free module over $S^G$ and 
the ring of quasiinvariants is Cohen-Macaulay and Gorenstein
(i\hspace{-.1em}i) there is a isomorphism from the dual space of the quotient of quasiinvariants by $S^G$ 
to the space of $m$-harmonic polynomials
(i\hspace{-.1em}i\hspace{-.1em}i) the Hilbert series of the quotient of the
quasiinvariants by $S^G_+$
is equal to that of $m$-harmonic polynomials.

Let $I_2(N)$ be the dihedral group of regular N-gon.
In \cite{feigin}, Feigin and Veselov considered quasiinvariants for $I_2(N)$ 
for a constant multiplicity. 
Since $I_2(N)$ is of rank $2$, quasiinvariants are expressed as essentially one variable.
Feigin and Veselov gave generators over $S^{I_2(N)}$ by a direct calculation. 
In \cite{feigin2}, Feigin studied quasiinvariants for $I_2(N)$ for any non-negative multiplicity.
He gave a free basis of quasiinvariants over $S^{I_2(N)}$
using the above mentioned results of Etingof and Ginzburg. 
An explicit description of basis of the quotient of quasiinvariants for $S_3$ 
is contained in \cite{feigin}. 
Another description is given in \cite{bandlow2}.
In \cite{felder}, for $S_n$ Felder and Veselov provide integral expressions 
for the lowest degree non-symmetric quasiinvariant polynomials (the degree $nm+1$).
However, for any integer $n\geq 4$ a basis of the quotient of quasiinvariants for $S_n$ is not known. 

In this paper, we consider quasiinvariants for $S_n$. 
In this case, $m$ is a non-negative integer.
We denote these quasiinvariants by ${\bf QI_m}$ and by $\Lambda_n$ the space of symmetric polynomials.
We define ${\bf QI_m}^*$ as the quotient of ${\bf QI_m}$ 
by the ideal generated by the homogeneous symmetric polynomials of positive degree.

In \cite{bandlow}, Bandlow and Musiker showed that ${\bf QI_m}$ has
a decomposition indexed by standard tableaux. 
Each component has a $\Lambda_n$ module structure. 
This decomposition can be extended to that of ${\bf QI_m}^*$.
They constructed explicit basis of the submodules of ${\bf QI_m}^*$ 
indexed by standard tableaux of shape $(n-1,1)$. 

In this paper, we extend the result in \cite{bandlow}.
We construct basis of the submodules of ${\bf QI_m}^*$ 
indexed by standard tableaux of shape $(n-k+1,1^{k-1})$ (a hook)
(see Theorem.\ref{basis:n-k+1}). 
The elements of our basis are expressed as determinants 
of a matrix with entries similar to elements of basis introduced in \cite{bandlow}.
We also show that our basis is a free basis of the submodule of ${\bf QI_m}$ 
indexed by a hook $(n-k+1,1^{k-1})$ over $\Lambda_n$ (Corollary.\ref{cor:n-k333}). 

We also show how the operator $L_m$
acts on our basis. In \cite{feigin}, it is proved that the operator $L_m$ preserves ${\bf QI_m}$.
In \cite{bandlow}, it is obtained explicit formulas of the action of $L_m$ on their basis. 
We extend this formulas to that of our basis (Theorem.\ref{thmLm}). 

\section{Preliminaries}
\subsection{Symmetric group and Young diagram}
We denote $\mathbb{Q}[x_1,x_2,\cdots ,x_n]$ by $K_n$ 
and the symmetric group on $\{1,2,\cdots ,n\}$ by $S_n$. 
For the finite set $X$, we denote the symmetric group on $X$ by $S_X$.

The symmetric group $S_n$ acts on $K_n$ by
\begin{equation*}
\sigma P(x_1,\cdots ,x_n)=P(x_{\sigma (1)},\cdots ,x_{\sigma (n)})\ \ \ \sigma \in S_n.
\end{equation*}
A polynomial $P(x_1,x_2,\cdots ,x_n)$ is called a symmetric polynomial when for any $\sigma \in S_n$ 
$P(x_1,x_2,\cdots ,x_n)$ satisfies
\begin{equation*}
\sigma P(x_1,\cdots ,x_n)=P(x_1,\cdots ,x_n).
\end{equation*}
We denote by $\Lambda_n$ the subspace spanned by symmetric polynomials 
and by $\Lambda_n^d$ the subspace of $\Lambda_n$ spanned by homogeneous polynomials of degree $d$. 
We set $\Lambda_n^d=\{ 0 \}$ if $d<0$.
The $i$-th elementary symmetric polynomial is denoted by $e_i$.
For a partition $\nu=(\nu_1,\nu_2,\cdots)$, we define $e_{\nu}=\prod_{i}e_{\nu_i}$.
A basis of $\Lambda_n$ is given by $\{ e_{\nu} \}$. 

The group ring $S_n$ over $\mathbb{Q}$ is denoted by $\mathbb{Q}S_n$.
The action of $S_n$ on $K_n$ is naturally extended to that of $\mathbb{Q}S_n$.
We define the elements of $\mathbb{Q}S_n$.
For a subgroup $H$ of $S_n$, 
we define $[H],[H]'$ by
\begin{gather*}
[ H ] = \sum_{\sigma \in H} \sigma \\
[ H ]' = \sum_{\sigma \in H} sgn(\sigma) \sigma . 
\end{gather*}

Let $\lambda=(\lambda_1, \lambda_2,\cdots)$ be a partition. 
When $\lambda$ is a partition of a positive integer $n$, we denote this by $\lambda \vdash n$.
We define $l(\lambda)=\sharp \{ i| \lambda_i\neq 0 \}$ and $|\lambda|=\sum_{i}\lambda_i$. 
They are called the length and the size of $\lambda$ respectively. 

For a partition $\lambda$, the Young diagram of shape $\lambda$ 
is a diagram such that its $i$-th row has $\lambda_i$ boxes 
and it is arranged in left-justified rows. 
For example, the Young diagram of shape $(4,3,1)$ is 
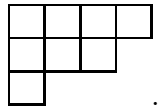
\begin{equation*}
\begin{Young}
&&& \cr
&& \cr
 \cr
\end{Young}
.
\end{equation*}

We denote by $(i,j)$ a box on the $(i,j)$-th position of the diagram. 
For instance, the box $(2,3)$ of the Young diagram of shape $(4,3,1)$ is 
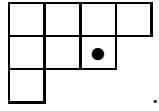
\begin{equation*}
\begin{Young}
&&& \cr
&&$\bullet$ \cr
 \cr
\end{Young}
.
\end{equation*}
We identify Young diagram of shape $\lambda$ with a partition $\lambda$. 

Let $k$, $n$ be integers such that $k\geq 2$ and $n\geq k$.
We define $\eta(n,k)=(n-k+1,1^{k-1})$. We have $l(\eta(n,k))=k$ and $|\eta(n,k)|=n$.
We call $\eta(n,k)$ (also the tableau of shape $\eta(n,k)$) the hook.

For $\lambda \vdash n$, we define the arm length $a(i,j)$ for box $(i,j)\in\lambda$ as 
\begin{equation*}
a(i,j)=\sharp \{ (i,l)\mid j<l,\ (i,l)\in \lambda \}.
\end{equation*}
We also define the leg length $l(i,j)$ for box $(i,j)$ as 
\begin{equation*}
l(i,j)=\sharp \{ (k,j)\mid i<k,\ (k,j)\in \lambda \}.
\end{equation*} 
We define $h(i,j)=a(i,j)+l(i,j)+1$ called the hook length for box $(i,j)\in\lambda$.

A $tableau$ of shape $\lambda$ is a diagram filled in each box of $\lambda$ with a positive integer. 
In this paper, we assume that entries of boxes are different each other. 
For a tableau $D$, we denote by $D_{i,j}$ the entry in the box $(i,j)$ of $D$. 
We define
\begin{equation*}
mem(D)=\{ D_{i,j} \mid (i,j)\in \lambda \}.
\end{equation*}

A tableau $T$ is called a standard tableau 
if $T$ satisfies $mem(T)=\{ 1,2,\cdots ,n \}$ and 
\begin{equation*}
T_{i,j}<T_{k,j},\ T_{i,j}<T_{i,l}\ \ \ i<k,j<l.
\end{equation*}
We denote by $ST(\lambda)$ the set of all standard tableaux of shape $\lambda$ 
and by $ST(n)$ the set of all standard tableaux with $n$ boxes.
 
For a tableau  $D$ of shape $\lambda$, we define
\begin{gather*}
C(D)=[ \{ \sigma\in S_{mem(D)}\mid \sigma\ preserves\ each\ column\ of\ D \} ]' \\
R(D)=[ \{ \sigma \in S_{mem(D)} \mid \sigma\ preserves\ each\ row\ of\ D \} ] \\
f_{\lambda} = |ST(\lambda)| \\
\gamma_D =\frac{f_{\lambda} C(D)R(D)}{n!} \\
V_D =\prod_{(i,j)\in C_D} (x_i-x_j)
\end{gather*}
where $C_D=\{ (i,j) \mid i<j$ and $i,j$ are the entries in same column of $D \}$.
The element $\gamma_D \in\mathbb{Q}S_{mem(D)}$ satisfies $\gamma_D^2=\gamma_D$.

\begin{definition}
{\rm
Let $s_1, s_2, \cdots ,s_n $ be mutually distinct positive integers. \\ 
(1)
We denote by $D(s_1, s_2, \cdots ,s_k;s_1,s_{k+1},$ $\cdots ,s_n)$ the 
tableau of shape $\eta(n,k)$ such that the entries in the first column and in the first row are
$s_1, s_2, \cdots ,s_k$ and $s_1, s_{k+1},\cdots ,s_n$ in order, respectively. \\
(2)
A tableau $D(s_1, s_2, \cdots ,s_k;s_1,s_{k+1}, \cdots ,s_n)$ 
is a standard tableau of shape $\eta(n,k)$ if and only if the following holds:
\begin{eqnarray*}
s_1, s_2, \cdots ,s_n\ {\rm is\ a\ permutation\ of\ } 1,2,\cdots ,n \\
s_1=1, s_2\leq \cdots \leq s_k,\ 
s_{k+1}\leq \cdots \leq s_n.
\end{eqnarray*}
Then we simply write $D(s_1, s_2, \cdots ,s_k;s_1,s_{k+1}, \cdots ,s_n)$ as $T(1,s_2,\cdots ,s_k)$. \\
(3)Let $i$ be a integer such that $1\leq i\leq k$ (resp. $k+1\leq i\leq n$). 
We set $D=D(s_1, s_2, \cdots ,s_k;s_1,s_{k+1},$ $\cdots ,s_n)$.
We define 
\begin{eqnarray*}
&&D^{s_i}=D(s_1,\cdots ,s_{i-1},s_{i+1},\cdots ,s_k;s_1,s_{k+1}, \cdots ,s_n) \\
&&(resp.\  D^{s_i}=D(s_1,\cdots ,s_k;s_1,s_{k+1}, \cdots 
,s_{i-1},s_{i+1},\cdots ,s_n)).
\end{eqnarray*}
}
\end{definition}

For example, a standard tableau $T(1,3,4)=D(1,3,4;1,2,5,6)$ of shape $(4,1,1)$ is
\begin{equation*}
\begin{Young}
$1$&$2$&$5$&$6$ \cr
$3$ \cr
$4$ \cr
\end{Young}
.
\end{equation*}
The tableau $T(1,3,4)^1$ is 
\begin{equation*}
\begin{Young}
$3$&$2$&$5$&$6$ \cr
$4$ \cr
\end{Young}
\end{equation*}
and $T(1,3,4)^2$ is
\begin{equation*}
\begin{Young}
$1$&$5$&$6$ \cr
$3$ \cr
$4$ \cr
\end{Young}
.
\end{equation*}

We have the following propositions.

\begin{proposition}[\cite{bandlow}]
For any $f=\sum_{\sigma \in S_n} f_{\sigma} \sigma \in \mathbb{Q}S_n$, 
$P\in \Lambda_n$ and $Q \in K_n$, 
we have $f(PQ)=Pf(Q)$.
\end{proposition}

\begin{proposition}[\cite{bandlow}]
Let $i_1,i_2,\cdots ,i_n$ be a permutation of ${1, 2, \cdots ,n}$. 
Then $[S_n]$ and $[S_n]'$ are expressed as follows:

\begin{gather*}
[S_n]=\bigl(1+ (i_1,i_n)+\cdots +(i_{n-1},i_n) \bigr) \cdots \bigl(1+(i_1,i_3)+(i_2,i_3) 
\bigr)\bigl(1+ (i_1,i_2) \bigr)\\
[S_n]'=\bigl(1- (i_1,i_n)-\cdots -(i_{n-1},i_n) \bigr) \cdots \bigl(1-(i_1,i_3)-(i_2,i_3) 
\bigr)\bigl(1- (i_1,i_2) \bigr).
\end{gather*}
\label{group ring}
\end{proposition}

\subsection{The quasiinvariants for $S_n$}

We first recall results in \cite{bandlow}.

\begin{lemma}[\cite{bandlow}]
The quasiinvariants ${\bf QI_m}$ has following decomposition:
\begin{equation*}
{\bf QI_m}=\bigoplus_{T\in ST(n)}\gamma_{T}({\bf QI_m}).
\end{equation*}

$\gamma_{T}({\bf QI_m})$ has following description:

\begin{equation}
\gamma_{T}({\bf QI_m})=\gamma_{T}(K_n)\cap V_{T}^{2m+1}K_n.
\label{eq:gammaT}
\end{equation}
\label{basis:lemn-1}
\end{lemma}

For $\lambda\vdash n$, the vector space $\displaystyle{\bigoplus_{T\in ST(\lambda)}\gamma_{T}({\bf QI_m})}$ is 
called the $\lambda -isotypic$ $component$ of ${\bf QI_m}$.

Let $K$ be a polynomial ring. 
We denote by $K[i]$ the subspace spanned by homogeneous polynomials of degree $i$ in $K$.
The Hilbert series of $K$ is defined as a formal power series 
$\displaystyle{\sum_{i=0}^{\infty}dim(K[i])t^i}$. We denote it by $H(K,t)$. 
 
We denote ${\bf QI_m}/\langle e_1,\cdots,e_n\rangle$ by ${\bf QI_m}^*$.
For $[f]\in {\bf QI_m}^*$, we define the degree of $[f]$ as the minimal degree in the class $[f]$.
In \cite{etingof} and \cite{felder},
the Hilbert series of ${\bf QI_m}^*$ 
is given by as follows:
\begin{theorem}[\cite{etingof}, \cite{felder}]
\begin{equation}
H({\bf QI_m}^*, t)=n!t^{mn(n-1)/2}\sum_{\lambda \vdash n}
\prod_{(i,j)\in \lambda}\prod_{k=1}^nt^{w(i,j;m)}
\frac{1-t^k}{h(i,j)(1-t^{h(i,j)})}
\end{equation}
where we set $w(i,j;m)=m(l(i,j)-a(i,j))+l(i,j)$.

In particular, for $T\in ST(\lambda)$ 
the Hilbert series of $\gamma_T({\bf QI_m}^*)$ is given as follows:
\begin{equation}
H(\gamma_T({\bf QI_m}^*); t)=t^{mn(n-1)/2}\prod_{(i,j)\in \lambda}\prod_{k=1}^nt^{w(i,j;m)}
\frac{1-t^k}{1-t^{h(i,j)}}.
\label{hilbert:gammaT}
\end{equation}
\end{theorem}

Let $s_1, s_2, \cdots ,s_n $ be mutually distinct positive integers. 
We set $D=D(s_1,s_2;s_1,s_3,\cdots ,s_n)$. 
We define the following polynomial in $\mathbb{Q}[x_{s_1},x_{s_2},\cdots ,x_{s_n}]$:
\begin{equation}
Q_{D}^{l;m}= \int_{x_{s_1}}^{x_{s_2}}t^l\prod_{i=1}^n(t-x_{s_i})^m dt.
\label{basis:n-1}
\end{equation}

Recall that we define $\eta(n,k)=(n-k+1,1^{k-1})$. 
In \cite{bandlow}, J.Bandlow and G.Musiker found an explicit 
basis of $\gamma_{T}({\bf QI_m}^*)$ when $T\in ST( \eta(n,2) )$. 
\begin{theorem}[\cite{bandlow}]
Let $T\in ST( \eta(n,2) )$. 
$\bigl\{ Q_{T}^{0;m}, Q_{T}^{1;m}, \cdots ,Q_{T}^{n-2;m} \bigr\}$ is 
a basis of $\gamma_{T}({\bf QI_m}^*)$. 
\label{basis:thn-1}
\end{theorem}
\begin{remark}
{\rm In \cite{bandlow}, it is shown that $Q_{T}^{l;m}$ is divisible by $V_T=(x_1-x_j)^{2m+1}$.
We can similarly show that $Q_{D}^{l;m}$ is divisible by
$V_D=(x_{s_1}-x_{s_2})^{2m+1}$.
}
\label{remark:VTn-1}
\end{remark}

Let $f\in \mathbb{Q}[x_{s_1},x_{s_2},\cdots ,x_{s_n}]$. We denote by $deg_{x_{s_i}}(f)$ the degree of $f$ 
as the polynomial in $x_{s_i}$.
For a homogeneous polynomial $g$, we define $deg(g)$ as the degree of $g$.

The polynomials $Q_{D}^{l;m}$ have the following properties, 
which we will use to to show Proposition.\ref{property:n-k}.  
\begin{proposition}
Let $s_1, s_2, \cdots ,s_n $ be mutually distinct positive integers.
Let $l$ be a non-negative integer and 
take a tableau $D=D(s_1,s_2;s_1,s_3,\cdots ,s_n)$ of shape $\eta(n,2)$.
 
$\displaystyle{Q_{D}^{l;m}}$ is a homogeneous polynomial of degree $nm+l+1$ and satisfies following 
properties. \\
{\rm(1)}
$\displaystyle{Q_{D}^{l;m}}$ is symmetric in $x_{s_3},\cdots ,x_{s_n}$ 
and anti-symmetric in $x_{s_1},x_{s_2}$. \\
{\rm(2)}
We have $\displaystyle{deg_{x_{s_1}}(Q_{D}^{l;m}})=nm+l+1$. The leading coefficient 
of $\displaystyle{Q_{D}^{l;m}}$ in $x_{s_1}$ is 
$\displaystyle{\frac{(-1)^{m+1}m!}{\prod_{s=0}^m(mn+l+1-s)}}$. \\
{\rm(3)}
Let $i\in \{ 1, \cdots ,n \} \backslash \{1,2\}$. 
We have $\displaystyle{deg_{x_{s_i}}(Q_{D}^{l;m}})=m$.
The leading coefficient of $\displaystyle{Q_{D}^{l;m}}$ in $x_{s_i}$ 
is equal to $\displaystyle{(-1)^mQ_{D^{s_i}}^{l;m}}$. 

\label{property:n-1}
\end{proposition}

\begin{proof}
{\rm 
We show the case $D=T(1,2)$ since the proofs of other cases are similar. 
We set $T=T(1,2)$. \\
(1)
It follows from the fact that $t^l\prod_{i=1}^n(t-x_i)^m$ 
is symmetric in $x_1,x_2,\cdots ,x_n$. \\
(2)We show this statement by induction on $m$. 

When $m=0$, $Q_{T}^{l;0}$ is $\frac{1}{l+1}(x_j^{l+1}-x_1^{l+1})$, the statement holds. 

When $m\geq 1$, assume that the statement holds for all numbers less than $m$. 
In \cite{bandlow}, the polynomial $Q_{T}^{l;m}$ is expressed as:
\begin{equation}
Q_{T}^{l;m}=\sum_{i=0}^n(-1)^i e_i Q_{T}^{n+l-i;m-1}.
\label{bandlowexpand}
\end{equation}

By induction assumption on $m$, we have $deg_{x_{s_1}}(Q_{T}^{n+l-i;m-1})=nm+l-i+1$. 
From (\ref{bandlowexpand}), we have $deg_{x_1}(Q_{T}^{l;m})=nm+l+1$ and 
the term with the highest degree are in 
$e_0Q_{T}^{n+l;m-1}-e_1Q_{T}^{n+l-1;m-1}$.
The leading coefficient of $Q_{T}^{l;m}$ in $x_1$ is 
\begin{eqnarray*}
&&\frac{(-1)^m(m-1)!}{\prod_{s=0}^{m-1}(mn+l+1-s)} 
-\frac{(-1)^m(m-1)!}{\prod_{s=0}^{m-1}(mn+l-s)} \\
&=&\frac{(-1)^{m+1}m!}{\prod_{s=0}^m(mn+l+1-s)}. 
\end{eqnarray*}
\\
(3)Expanding $(t-x_i)^m$ in $Q_{T}^{l;m}$, we have

\begin{equation*}
Q_{T}^{l;m}=\sum_{s=0}^m(-1)^s \binom{m}{s} Q_{T^i}^{l;m}x_i^s.
\end{equation*}

Thus propositions is proved. $\Box$}

\end{proof}

As a corollary of this proposition, we have $Q_{D}^{l;m}\neq 0$ 
when $D$ is a tableau of shape $\eta(n,2)$.

\section{A basis for the isotypic component of shape $(n-k+1,1^{k-1})$}
We give a basis for the $\eta(n,k)$-isotypic component.
Let $s_1, s_2, \cdots ,s_n $ be mutually distinct positive integers.  
Throughout this section, we set  $D=D(s_1,\cdots ,s_k;s_1,s_{k+1},\cdots ,s_n)$ and 
$T=T(1,2,\cdots ,k)$.

\begin{definition}
{\rm 
(1)Let $p$ be a non-negative integer. 
For $i,j$ such that $1\leq i<j \leq k$, we define a polynomial $R_{D;s_i,s_j}^{p;m}$ in 
$\mathbb{Q}[x_{s_1},x_{s_2},\cdots ,x_{s_n}]$ as
\begin{equation}
R_{D;s_i,s_j}^{p;m}=\int_{x_{s_i}}^{x_{s_j}}t^p\prod_{l=1}^n(t-x_{s_l})^m dt.
\end{equation} \\
(2)Let $k$ be an integer such that $k\geq 2$. Take a partition $\mu =(\mu_1,\mu_2,\cdots ,\mu_{k-1})$ 
such that $\mu_1>\mu_2>\cdots >\mu_{k-1} \geq 0$.
We define a polynomial $Q_{D}^{\mu;m}$ in $\mathbb{Q}[x_{s_1},x_{s_2},\cdots ,x_{s_n}]$ as follows:

\begin{equation}
Q_{D}^{\mu;m}=
\begin{vmatrix}
R_{D;s_1,s_2}^{\mu_1;m}     & R_{D;s_1,s_2}^{\mu_2;m}     & \cdots & R_{D;s_1,s_2}^{\mu_{k-1};m} \\
R_{D;s_2,s_3}^{\mu_1;m}     & R_{D;s_2,s_3}^{\mu_2;m}     & \cdots & R_{D;s_2,s_3}^{\mu_{k-1};m} \\
  \vdots               &         \vdots         & \ddots &        \vdots          \\
R_{D;s_{k-1},s_k}^{\mu_1;m}   & R_{D;s_{k-1},s_k}^{\mu_2;m}   & \cdots & R_{D;s_{k-1},s_k}^{\mu_{k-1};m}
\end{vmatrix} 
.
\label{basis:n-k}
\end{equation}
}
We denote the empty sequence by $\emptyset$. 
When $k=1$, $\mu$ is the empty sequence $\emptyset$. We set $Q_{D}^{\emptyset;m}=1$.
We simply write $Q_{D}^{m}$ as $Q_{D}^{\emptyset;m}$.
\end{definition}
\begin{remark}
{\rm
Setting $D'=D(s_1,s_2;s_1,s_3,\cdots ,s_n)$, 
we have $R_{D;s_1,s_2}^{p;m}=Q_{D'}^{p;m}$. 
}
\end{remark}

The polynomials $Q_{D}^{\mu;m}$ have the following properties, 
which we will use to show our main results.
\begin{proposition}
Let $s_1, s_2, \cdots ,s_n $ be mutually distinct positive integers.  
We set $D=D(s_1,\cdots ,s_k;s_1,s_{k+1},\cdots ,s_n)$.
Let $\mu =(\mu_1,\mu_2,\cdots ,\mu_{k-1})$ be a partition  
such that $\mu_1>\mu_2>\cdots >\mu_{k-1} \geq 0$. 

Then, $\displaystyle{Q_{D}^{\mu;m}}$ satisfies the following. \\
{\rm(1)}$\displaystyle{Q_{D}^{\mu;m}}$ is symmetric in $x_{s_{k+1}},x_{s_{k+2}},\cdots ,x_{s_n}$ 
and anti-symmetric in $x_{s_1},x_{s_2},\cdots$ $,x_{s_k}$.
In particular, $\displaystyle{Q_{D}^{\mu;m}}$  is divisible by $\displaystyle{V_D^{2m+1}}$. \\
{\rm(2)}We have $\displaystyle{deg_{x_{s_1}}(Q_{D}^{\mu;m}})=(n+k-2)m+\mu_1+1$.
The leading coefficient of $\displaystyle{Q_{D}^{\mu;m}}$ in $x_{s_1}$ is 
\begin{equation*} 
\frac{(-1)^{(k-1)m+1}m!}{\prod_{s=0}^m(mn+\mu_1+1-s)}
Q_{D^{s_1}}^{(\mu_2,\cdots ,\mu_{k-1});m}. 
\end{equation*}

In particular, we have $\displaystyle{deg(Q_{D}^{\mu;m}})=(k-1)nm+|\mu|+k-1$. \\
{\rm(3)}We have $\displaystyle{deg_{x_{k+1}}(Q_{D}^{\mu;m}})=(k-1)m$. 
The leading coefficient of $\displaystyle{Q_{D}^{\mu;m}}$ in $x_{k+1}$ is 
$\displaystyle{(-1)^{(k-1)m}Q_{D^{s_{k+1}}}^{\mu;m}}$. \\
{\rm(4)}Polynomial $\displaystyle{Q_{D}^{\mu;m}}$ is invariant under $\displaystyle{\gamma_D}$. 
\label{property:n-k}
\end{proposition}
\begin{proof}
{\rm We show the case $D=T$. The proofs of other cases are similar. 
(1)From Prop.$\ref{property:n-1}$(1),
it follows that $Q_{T}^{\mu;m}$ is symmetric in $x_{k+1},x_{k+2},\cdots ,x_n$. 

Adding the first row to the second row, we get 
\begin{equation*}
Q_{T}^{\mu;m}=
\begin{vmatrix}
R_{T;1,2}^{\mu_1;m}     & R_{T;1,2}^{\mu_2;m}     & \cdots & R_{T;1,2}^{\mu_{k-1};m} \\
R_{T;1,3}^{\mu_1;m}     & R_{T;1,3}^{\mu_2;m}     & \cdots & R_{T;1,3}^{\mu_{k-1};m} \\
  \vdots               &         \vdots         & \ddots &        \vdots          \\
R_{T;k-1,k}^{\mu_1;m}   & R_{T;k-1,k}^{\mu_1;m}   & \cdots & R_{T;k-1,k}^{\mu_{k-1};m}
\end{vmatrix} 
. 
\end{equation*}
Repeating this process, we get 

\begin{equation}
Q_{T}^{\mu;m}=
\begin{vmatrix}
R_{T;1,2}^{\mu_1;m}     & R_{T;1,2}^{\mu_2;m}     & \cdots & R_{T;1,2}^{\mu_{k-1};m} \\
R_{T;1,3}^{\mu_1;m}     & R_{T;1,3}^{\mu_2;m}     & \cdots & R_{T;1,3}^{\mu_{k-1};m} \\
  \vdots               &         \vdots         & \ddots &        \vdots          \\
R_{T;1,k}^{\mu_1;m}     & R_{T;1,k}^{\mu_2;m}     & \cdots & R_{T;1,k}^{\mu_{k-1};m}
\end{vmatrix} 
. \label{det399}
\end{equation}
Thus, $Q_{T}^{\mu;m}$ is anti-symmetric in $x_2,\cdots ,x_k$. 
We can show that $Q_{T}^{\mu;m}$ is anti-symmetric in $x_1, x_3,\cdots ,x_k$ and 
$x_1, x_2,x_4\cdots ,x_k$ in the similar way. Thus the first statement holds. 

From Remark.\ref{remark:VTn-1} and (\ref{det399}), $Q_{T}^{\mu;m}$ is divisible by 
$\prod_{s=2}^n(x_1-x_s)^{2m+1}$. 
Using this proposition $(1)$ we see $Q_{T}^{\mu;m}$ is also divisible by $V_T^{2m+1}$. \\
(2)
We see $Q_{T}^{\mu;m}$ as a polynomial in $x_1$.
From Prop.$\ref{property:n-1}$ (2),(3), 
the term of the degree of $Q_{T}^{\mu;m}$ is in
$R_{T;1,2}^{\mu_1;m}R_{T;2,3}^{\mu_2;m}\cdots R_{T;k-1,k}^{\mu_k;m}$. 
We use Prop.$\ref{property:n-1}$ (2),(3) again, therefore the statement holds. \\
(3)
From Prop.$\ref{property:n-1}$ (3), 
the leading coefficient of $x_{k+1}$ is

\begin{equation}
\begin{vmatrix}
(-1)^mR_{T^{k+1};1,2}^{\mu_1;m} & (-1)^mR_{T^{k+1};1,2}^{\mu_2;m} & \cdots & (-1)^mR_{T^{k+1};1,2}^{\mu_k;m} \\
(-1)^mR_{T^{k+1};2,3}^{\mu_1;m} & (-1)^mR_{T^{k+1};2,3}^{\mu_2;m} & \cdots & (-1)^mR_{T^{k+1};2,3}^{\mu_k;m} \\
  \vdots           &      \vdots        & \ddots &    \vdots          \\
(-1)^mR_{T^{k+1};k-1,k}^{\mu_1;m} & (-1)^mR_{T^{k+1};k-1,k}^{\mu_2;m} & \cdots & (-1)^mR_{T^{k+1};k-1,k}^{\mu_k;m}
\end{vmatrix}
. 
\label{eq:001}
\end{equation}

The polynomial (\ref{eq:001}) is equal to 
$(-1)^{(k-1)m}Q_{T^{k+1}}^{\mu;m}$. \\
(4)
To prove (4), we define the following notion.

For positive integers $i,j$ such that $i\neq j$, 
we define a tableau $(i,j)D$ as follows.
When $i,j \not\in mem(D)$, we define $(i,j)D=D$. 
When $i\in mem(D)$ and $j\not\in mem(D)$, 
$(i,j)D$ is a tableau obtained by replacing the entry $i$ in $D$ with $j$.
When $i,j\in mem(D)$,  is a tableau obtained by interchanging the entry $i$ and $j$ in $D$.

Using Prop.$\ref{group ring}$, $\gamma_T$ is 
\begin{equation*}
\frac{1}{n(n-k)!(k-1)!}\bigl\{ 1-\sum_{s=2}^k(1,s) \bigr\} [S_{\{ 2,3,\cdots ,k \}}]'
\bigl\{ 1+\sum_{s=k+1}^n(1,s) \bigr\} [S_{\{ k+1,\cdots,n \}}].
\end{equation*}

Therefore from (1), 
\begin{equation*}
\gamma_T(Q_{T}^{\mu;m})=\frac{1}{n}\bigl\{ kQ_{T}^{\mu;m} 
+\sum_{s=k+1}^n \{ 1-(1,2)-\cdots -(1,k)\} Q_{(1,s)T}^{\mu;m}\bigr\}.
\end{equation*}

We consider the sum $\displaystyle{\sum_{s=k+1}^n \{ 1-(1,2)-\cdots -(1,k)\} Q_{(1,s)T}^{\mu;m}}$.

\begin{align*}
&\sum_{s=k+1}^n \{ 1-(1,2)-(1,3)-\cdots -(1,k)\} Q_{(1,s)T}^{\mu;m} \\
=&\sum_{s=k+1}^n \{ Q_{(1,s)T}^{\mu;m}+Q_{(2,s)T}^{\mu;m}
+Q_{(3,s)T}^{\mu;m}+\cdots +Q_{(k,s)T}^{\mu;m} \}.
\end{align*}

Consider the sum $Q_{(1,s)T}^{\mu;m}+Q_{(2,s)T}^{\mu;m}$. 
By definition, we have
\begin{gather*}
Q_{(1,s)T}^{\mu;m}+Q_{(2,s)T}^{\mu;m}
\\
=
\begin{vmatrix}
R_{T;s,2}^{\mu_1;m}     & R_{T;s,2}^{\mu_2;m}     & \cdots & R_{T;s,2}^{\mu_{k-1};m} \\
R_{T;2,3}^{\mu_1;m}     & R_{T;2,3}^{\mu_2;m}     & \cdots & R_{T;2,3}^{\mu_{k-1};m} \\
  \vdots               &         \vdots         & \ddots &        \vdots          \\
R_{T;k-1,k}^{\mu_1;m}   & R_{T;k-1,k}^{\mu_2;m}   & \cdots & R_{T;k-1,k}^{\mu_{k-1};m}
\end{vmatrix} 
+
\begin{vmatrix}
R_{T;1,s}^{\mu_1;m}     & R_{T;1,s}^{\mu_2;m}     & \cdots & R_{T;1,s}^{\mu_{k-1};m} \\
R_{T;s,3}^{\mu_1;m}     & R_{T;s,3}^{\mu_2;m}     & \cdots & R_{T;s,3}^{\mu_{k-1};m} \\
  \vdots               &         \vdots         & \ddots &        \vdots          \\
R_{T;k-1,k}^{\mu_1;m}   & R_{T;k-1,k}^{\mu_2;m}   & \cdots & R_{T;k-1,k}^{\mu_{k-1};m}
\end{vmatrix} 
.
\end{gather*}

Adding the first row to the second row in the second determinant, we get
\begin{gather*}
Q_{(1,s)T}^{\mu;m}+Q_{(2,s)T}^{\mu;m}
\\
=
\begin{vmatrix}
R_{T;s,2}^{\mu_1;m}     & R_{T;s,2}^{\mu_2;m}     & \cdots & R_{T;s,2}^{\mu_{k-1};m} \\
R_{T;s,3}^{\mu_1;m}     & R_{T;s,3}^{\mu_2;m}     & \cdots & R_{T;s,3}^{\mu_{k-1};m} \\
  \vdots               &         \vdots         & \ddots &        \vdots          \\
R_{T;k-1,k}^{\mu_1;m}   & R_{T;k-1,k}^{\mu_2;m}   & \cdots & R_{T;k-1,k}^{\mu_{k-1};m}
\end{vmatrix} 
+
\begin{vmatrix}
R_{T;1,s}^{\mu_1;m}     & R_{T;1,s}^{\mu_2;m}     & \cdots & R_{T;1,s}^{\mu_{k-1};m} \\
R_{T;s,3}^{\mu_1;m}     & R_{T;s,3}^{\mu_2;m}     & \cdots & R_{T;s,3}^{\mu_{k-1};m} \\
  \vdots               &         \vdots         & \ddots &        \vdots          \\
R_{T;k-1,k}^{\mu_1;m}   & R_{T;k-1,k}^{\mu_2;m}   & \cdots & R_{T;k-1,k}^{\mu_{k-1};m}
\end{vmatrix}  
.
\end{gather*}
Adding the two term, we have

\begin{gather*}
Q_{(1,s)T}^{\mu;m}+Q_{(2,s)T}^{\mu;m}
\\
=
\begin{vmatrix}
R_{T;1,2}^{\mu_1;m}     & R_{T;1,2}^{\mu_2;m}     & \cdots & R_{T;1,2}^{\mu_{k-1};m} \\
R_{T;s,3}^{\mu_1;m}     & R_{T;s,3}^{\mu_2;m}     & \cdots & R_{T;s,3}^{\mu_{k-1};m} \\
  \vdots               &         \vdots         & \ddots &        \vdots          \\
R_{T;k-1,k}^{\mu_1;m}   & R_{T;k-1,k}^{\mu_2;m}   & \cdots & R_{T;k-1,k}^{\mu_{k-1};m}
\end{vmatrix} 
. 
\end{gather*}

Repeating this process, we get
\begin{equation*}
\{ 1-(1,2)-(1,3)-\cdots -(1,k)\} Q_{(1,s)T}^{\mu;m}=Q_{T}^{\mu;m}.
\end{equation*} 
Thus, proposition is proved. $\Box$}
\end{proof}

As a corollary of this proposition, we have
$Q_{T}^{\mu;m}\in \gamma_T({\bf QI_m})$ where $T\in ST(\eta(n,k))$. 

We introduce the following notions. 
\begin{definition}
{\rm 
Let $s,t,u$ be non-negative integers.
When $u\geq 1$, we set the subsets of partitions $P(s; t; u)$, $P(t; u)$ and $Q(s; t; u)$ as:
\begin{eqnarray*}
P(s; t; u)&=&\{ \lambda \mid |\lambda |=s, t\geq \lambda_1 >\lambda_2 >\cdots >\lambda_u 
\geq 0\} \\
Q(s; t; u)&=&P(s; t; u)\backslash P(s; t-1; u) \\
P(t; u)&=&\cup_{s\geq 0} P(s; t; u).
\end{eqnarray*}
When $u=0$, we set 
\begin{eqnarray*}
P(0; t; 0)&=&\{ \emptyset \} \\
P(t; 0)&=&\{ \emptyset \}.
\end{eqnarray*}
Let $l$ be a positive integer. We set $P(l; t; 0)$ as empty set.

We define $p(s; t; u)=\sharp P(s; t; u)$ and $q(s; t; u)=\sharp Q(s; t; u)$. 
}
\end{definition}
\begin{remark}
{\rm
Setting $\mu\in P(n-2;k-1)$ (resp. $\mu\in \cup_{s\geq 0} Q(s; n-2; k-1)$),
we have 
\begin{eqnarray*}
&&\frac{(k-1)(k-2)}{2}\leq |\mu| \leq (k-1)(n-k)+\frac{(k-1)(k-2)}{2} \\
&&\Bigl( resp.\ n-2+\frac{(k-2)(k-3)}{2}\leq |\mu|\leq (k-1)(n-k)+\frac{(k-1)(k-2)}{2} \Bigr).
\end{eqnarray*}
}
\label{remarkPQ}
\end{remark}

We have the following proposition.
\begin{proposition} 
Let $k$ be an integer such that $k\geq 2$. \\
{\rm(1)}Let $l$ be a integer such that $0\leq l\leq n-k-1$. Then, we have
\begin{equation*}
p\Bigl( l+\frac{(k-1)(k-2)}{2};n-3;k-1\Bigr) =p\Bigl( l+\frac{(k-1)(k-2)}{2};n-2;k-1 \Bigr).
\end{equation*} 
\\
{\rm(2)}Let $l$ be a integer such that $l\geq n-k$. Then, we have
\begin{eqnarray*}
&&p\Bigl( l+\frac{(k-1)(k-2)}{2};n-2;k-1\Bigr) \\
&=&p\Bigl( l+\frac{(k-1)(k-2)}{2};n-3;k-1\Bigr) \\
&+&p\Bigl( l+k-n+\frac{(k-2)(k-3)}{2};n-3;k-2 \Bigr).
\end{eqnarray*} \\
{\rm(3)}
Let $l\in \{0,1,\cdots ,k-2\}$. Then, we have
\begin{eqnarray*}
&&p\Bigl( (k-1)(n-k)+\frac{(k-1)(k-2)}{2}-l;n-2;k-1 \Bigr) \nonumber \\
&=&p\Bigl( (k-2)(n-k)+\frac{(k-2)(k-3)}{2}-l;n-3;k-2 \Bigr).
\end{eqnarray*}
\label{hilbert:n-k-prop}
\end{proposition}

\begin{proof}

{\rm
(1)By definition, we have
\begin{eqnarray*}
&&q\Bigl( l+\frac{(k-1)(k-2)}{2};n-2;k-1\Bigr) \\
&=&p\Bigl( l+\frac{(k-1)(k-2)}{2};n-2;k-1\Bigr)-p\Bigl( l+\frac{(k-1)(k-2)}{2};n-3;k-1\Bigr).
\end{eqnarray*}
Therefore we show $q\Bigl( l+\frac{(k-1)(k-2)}{2};n-2;k-1\Bigr)=0$. 

We have $l+\frac{(k-1)(k-2)}{2}\leq n-k-1+\frac{(k-1)(k-2)}{2}<n-2+\frac{(k-2)(k-3)}{2}$. 
From Remark.\ref{remarkPQ}, proposition follows.
\\
(2)
To prove (2), we show 
\begin{eqnarray*}
&&q\Bigl( l+\frac{(k-1)(k-2)}{2};n-2;k-1\Bigr) \\
&=&p\Bigl( l+k-n+\frac{(k-2)(k-3)}{2};n-3;k-2 \Bigr).
\end{eqnarray*}

By definition, we have
\begin{equation*}
q\Bigl( l+\frac{(k-1)(k-2)}{2};n-2;k-1\Bigr) =\sum_{s=0}^{n-3}q\Bigl( l+\frac{(k-1)(k-2)}{2}-n+2;s;k-2 \Bigr) .
\end{equation*}
We have $l+\frac{(k-1)(k-2)}{2}-n+2=l+k-n+\frac{(k-2)(k-3)}{2}$, therefore we get
\begin{eqnarray*}
&&q\Bigl( l+\frac{(k-1)(k-2)}{2};n-2;k-1\Bigr) \\
&=&\sum_{s=0}^{n-3}q\Bigl( l+k-n+\frac{(k-2)(k-3)}{2};s;k-2\Bigr).
\end{eqnarray*}
By definition, we obtain
\begin{eqnarray*}
&&\sum_{s=0}^{n-3}q\Bigl( l+k-n+\frac{(k-2)(k-3)}{2};s;k-2\Bigr) \\
&=&p\Bigl( l+k-n+\frac{(k-2)(k-3)}{2};n-3;k-2\Bigr).
\end{eqnarray*}
\\
(3)
By definition, we have
\begin{eqnarray*}
&&p\Bigl( (k-1)(n-k)+\frac{(k-1)(k-2)}{2}-l;n-2;k-1 \Bigr) \\
&=&\sum_{s=0}^{n-2}q\Bigl( (k-1)(n-k)+\frac{(k-1)(k-2)}{2}-l;s;k-1 \Bigr).
\end{eqnarray*}

From Remark.\ref{remarkPQ}, 
we have $q\Bigl( (k-1)(n-k)+\frac{(k-1)(k-2)}{2}-l;s;k-1 \Bigr)=0$ when $s\leq n-3$. 
Therefore we obtain 
\begin{eqnarray*}
&&p\Bigl( (k-1)(n-k)+\frac{(k-1)(k-2)}{2}-l;n-2;k-1 \Bigr) \\
&=&q\Bigl( (k-1)(n-k)+\frac{(k-1)(k-2)}{2}-l;n-2;k-1 \Bigr). 
\end{eqnarray*}
From (2), we have 
\begin{eqnarray*}
&&q\Bigl( (k-1)(n-k)+\frac{(k-1)(k-2)}{2}-l;n-2;k-1 \Bigr) \\
&=& p\Bigl( (k-1)(n-k)+\frac{(k-2)(k-3)}{2}-l+k-n;n-3;k-2 \Bigr) \\
&=& p\Bigl( (k-2)(n-k)+\frac{(k-2)(k-3)}{2}-l;n-3;k-2 \Bigr).\ \Box
\end{eqnarray*}
}
\end{proof}

We next consider the Hilbert series of $\gamma_T({\bf QI_m}^*)$.
To simplified notions, we define $p_{s,n-2,k-1}=p\Bigl( s+\frac{(k-1)(k-2)}{2};n-2;k-1 \Bigr)$.

Proposition.\ref{hilbert:n-k-prop} is written as:
\begin{eqnarray*}
&(1)&p_{l,n-3,k-1}=p_{l,n-2,k-1} \\
&(2)&p_{l,n-2,k-1}=p_{l,n-3,k-1}+p_{l+k-n,n-3,k-2} \\
&(3)&p_{(k-1)(n-k)-l,n-2,k-1}=p_{(k-2)(n-k)-l,n-3,k-2}. 
\end{eqnarray*}

\begin{lemma}
We have
\begin{equation}
H(\gamma_T({\bf QI_m}^*); t)=t^{(k-1)nm+\frac{k(k-1)}{2}}\sum_{s=0}^{(k-1)(n-k)}
p_{s,n-2,k-1} t^s.
\label{hilbetr:n-k-complete}
\end{equation}
\end{lemma}

\begin{proof}
{\rm From ($\ref{hilbert:gammaT}$), $H(\gamma_T({\bf QI_m}^*); t)$ is equal to
\begin{equation*}
t^{mn(n-1)/2}\prod_{(i,j)\in \lambda}\prod_{l=1}^n t^{m(l(i,j)-a(i,j))+l(i,j)}
\frac{1-t^l}{1-t^{h(i,j)}}.
\end{equation*}

For $2\leq i \leq n-k+1$ and $2\leq j \leq k$, we have 
\begin{eqnarray*}
&&a(1,1)=n-k,\ l(1,1)=k-1,\ h(1,1)=n \\
&&a(1,i)=n-k+1-i,\ l(1,i)=0,\ h(1,i)=n-k+2-i \\
&&a(j,1)=0,\ l(j,1)=k-j,\ h(j,1)=k-j+1. 
\end{eqnarray*}

Thus we have
\begin{equation*}
H(\gamma_T({\bf QI_m}^*); t) 
=t^{(k-1)nm+\frac{k(k-1)}{2}}\prod_{s=1}^{k-1}\frac{(1-t^{n-s})}{(1-t^s)}.
\end{equation*}

Therefore, we must show 
\begin{equation}
\prod_{s=1}^{k-1}\frac{(1-t^{n-s})}{(1-t^s)}
=\sum_{s=0}^{(k-1)(n-k)}p_{s,n-2,k-1}t^s.
\label{hilbetr:n-2sub}
\end{equation}

We show this by induction on $n$.

If $n=k$, then both of l.h.s and r.h.s are both equal to 1.

When $n\geq k+1$, we assume that ($\ref{hilbetr:n-2sub}$) holds with all numbers less than $n$.
We have the following identity:
\begin{equation*}
\prod_{s=1}^{k-1}\frac{(1-t^{n-s})}{(1-t^s)}=
\prod_{s=1}^{k-1}\frac{(1-t^{n-s-1})}{(1-t^s)}+t^{n-k}\prod_{s=1}^{k-2}\frac{(1-t^{n-s-1})}{(1-t^s)}.\\
\end{equation*}
By induction assumption, we obtain
\begin{eqnarray*}
&&\prod_{s=1}^{k-1}\frac{(1-t^{n-s-1})}{(1-t^s)}+t^{n-k}\prod_{s=1}^{k-2}\frac{(1-t^{n-s-1})}{(1-t^s)} \\
&=&\sum_{s=0}^{(k-1)(n-k-1)}p_{s,n-3,k-1}t^s 
+t^{n-k}\sum_{s=0}^{(k-2)(n-k)}p_{s,n-3,k-2}t^s .
\end{eqnarray*}

We can rewrite this as
\begin{eqnarray*}
&&\prod_{s=1}^{k-1}\frac{(1-t^{n-s})}{(1-t^s)} \\
&=&\sum_{s=(k-1)(n-k)-k+2}^{(k-1)(n-k)}p_{s-n+k,n-3,k-2}t^s \\
&+&\sum_{s=n-k}^{(k-1)(n-k-1)}\bigl( p_{s-n+k,n-3,k-2}+p_{s,n-3,k-1}\bigr) t^s \\
&+&\sum_{s=0}^{n-k-1}p_{s,n-3,k-1}t^s.
\end{eqnarray*}

Using Proposition.\ref{hilbert:n-k-prop} (2), we have
\begin{eqnarray*}
&&\sum_{s=n-k}^{(k-1)(n-k-1)}\bigl( p_{s-n+k,n-3,k-2}+p_{s,n-3,k-1}\bigr) t^s \\
&=&\sum_{s=n-k}^{(k-1)(n-k-1)}p_{s,n-2,k-1} t^s.
\end{eqnarray*}
From Proposition.\ref{hilbert:n-k-prop} (1) and (3), lemma holds.
$\Box$
}
\end{proof}

We state the main theorem in this paper. 
\begin{theorem}
$\bigl\{ Q_{T}^{\mu;m} \bigr\}_{\mu\in P(n-2;k-1)}$ is a basis of $\gamma_T({\bf QI_m}^*)$. 
\label{basis:n-k+1}
\end{theorem}

To simplified notions, we set
\begin{eqnarray*}
&&P_{s,n-2,k-1}=P\Bigl( s+\frac{(k-1)(k-2)}{2};n-2;k-1 \Bigr) \\
&&P_{n-2,k-1}=P\Bigl( n-2;k-1 \Bigr) \\
&&Q_{s,n-2,k-1}=Q\Bigl( s+\frac{(k-1)(k-2)}{2};n-2;k-1 \Bigr).
\end{eqnarray*}

We define following notions.

Let $X=\{ s_1,s_2,\cdots ,s_n \}$ be the set of $n$ positive integers.
We recall that $S_X$ is the symmetric group on $X$ 
and $S_X$ acts on $\mathbb{Q}[x_{s_1},x_{s_2},\cdots ,x_{s_n}]$ on the left.

We define $\Lambda_X$ as the subspace of $\mathbb{Q}[x_{s_1},x_{s_2},\cdots ,x_{s_n}]$ 
spanned by all polynomials which is invariant under $S_X$. 
We define $\Lambda_X^d$ as the subspace of $\Lambda_X$ spanned by homogeneous polynomials of degree $d$.
We define $\Lambda_X^d=\{ 0 \}$ if $d<0$.

Theorem.\ref{basis:n-k+1} follows from the following proposition. 
\begin{proposition} Let $D$ be a tableau of shape $\eta(n,k)$.
If
\begin{equation}
\sum_{\mu\in P(n-2;k-1)}f_{\mu}Q_{D}^{\mu;m}=0
\label{n-k:independents}
\end{equation}
where $f_{\mu}\in \Lambda_{mem(D)}$,
then all $f_{\mu}$ is equal to 0.
\label{proposition:n-k+1}
\end{proposition}

\begin{proof}
{\rm
We show this lemma by induction on the size of tableaux.

In the case $k=1$, (\ref{proposition:n-k+1}) is $fQ_{D}^{m}=0$ where $f\in \Lambda_{mem(D)}$.
Therefore proposition holds when $k=1$. We assume that $k\geq 2$.

We recall that $n\geq k$. 
When $n=2$, we have $k=2$. Then l.h.s of (\ref{n-k:independents}) is equal to 
$f_{0}Q_{D}^{0;m}$.
Therefore lemma holds when $n=2$.

Assume that (\ref{n-k:independents}) holds when size is less than $n$ for $n\geq 3$.
We show the case $D=T$ since the proofs of other cases are similar.

We recall that $\Lambda_n$ is a graded ring. Therefore we can decompose 
\begin{equation*}
f_{\mu}=\sum_{l\geq 0}f_{\mu,l}
\end{equation*}
where $f_{\mu,l}\in \Lambda_n^l$.
Thus, (\ref{n-k:independents}) is written as
\begin{equation}
\sum_{\mu\in P(n-2;k-1)}\sum_{l\geq 0}f_{\mu,l}Q_{T}^{\mu;m}=0
\label{n-k:1983}
\end{equation}
where $f_{\mu,l}\in \Lambda_n^l$.
We have $deg(Q_{T}^{\mu;m})=(k-1)nm+|\mu|+k-1$, 
therefore we obtain $deg(f_{\mu,l}Q_{T}^{\mu;m})=(k-1)nm+|\mu|+d+k-1$.

Thus, (\ref{n-k:1983}) is written as
\begin{equation}
\sum_{d\geq 0}\sum_{\mu\in P(n-2;k-1)}f_{\mu,d-(k-1)nm-|\mu|-k+1}Q_{T}^{\mu;m}=0.
\label{n-k:1984}
\end{equation}

Hence, for any $d$ we obtain
\begin{equation}
\sum_{\mu\in P(n-2;k-1)}f_{\mu,d-(k-1)nm-|\mu|-k+1}Q_{T}^{\mu;m}=0.
\label{n-k:1985}
\end{equation}

Fix $d$. Recall that the set $P_{s,n-2,k-1}$ is not empty set if $0\leq s\leq (k-1)(n-k)$.
Let $s$ be an integer such that $0\leq s\leq (k-1)(n-k)$ and 
take $\mu\in P_{s,n-2,k-1}$. Then we have $deg(Q_T^{\mu ;m})=(k-1)nm+\frac{k(k-1)}{2}+s$. 
We set $d'=d-(k-1)nm-\frac{k(k-1)}{2}$.
We express $f_{\mu,d'-s}$ as
\begin{equation*} 
\sum_{r=0}^{d'-s} \sum_{{\scriptstyle |\nu|=d'-s} \atop {\scriptstyle l(\nu)=r}}a^{\mu}_{r,\nu}e_{\nu}.
\end{equation*}

We recall that
\begin{eqnarray*}
&&P_{s,n-2,k-1}=P\Bigl( s+\frac{(k-1)(k-2)}{2};n-2;k-1 \Bigr) \\
&&P_{n-2,k-1}=P\Bigl( n-2;k-1 \Bigr) \\
&&Q_{s,n-2,k-1}=Q\Bigl( s+\frac{(k-1)(k-2)}{2};n-2;k-1 \Bigr).
\end{eqnarray*}
Therefore (\ref{n-k:1985}) is written as
\begin{equation}
\sum_{s=0}^{(k-1)(n-k)}\sum_{\mu \in P_{s,n-2,k-1}}
\sum_{r=0}^{d'-s}
\sum_{{\scriptstyle |\nu|=d'-s} \atop {\scriptstyle l(\nu)=r}}a^{\mu}_{r,\nu}e_{\nu}Q_{T}^{\mu;m}=0.
\label{the:n-2}
\end{equation}

We show $a^{\mu}_{r,\nu}=0$ for $r\geq 0$. We show this by induction on $r$.
To prove this, we consider the highest degree terms in $x_{k+1}$.

As a polynomial in $x_{k+1}$, the degree of l.h.s. of ($\ref{the:n-2}$) is $(k-1)m+d'$
and is in $a^{(k-2,k-3,\cdots ,0)}_{d',(1^{d'})}
e_{(1^{d'})}Q_{T}^{(k-2,k-3,\cdots ,0);m}$.
Hence we have $a^{(k-2,k-3,\cdots ,0)}_{d',(1^{d'})}=0$.

Thus using the following lemma, we complete the proof of Proposition.\ref{proposition:n-k+1}. 
\begin{lemma}
Let $k$ be an integer such that $k\geq 3$.
We assume that for each integer $l$ such that $2\leq l\leq n-1$ and each tableau $\eta(n-1,l)$,
the statements of Proposition.\ref{proposition:n-k+1} holds.

Let $r$ an integer such that $1\leq r \leq d'-1$.
If we have the following equation:
\begin{equation}
\sum_{s=0}^{(k-1)(n-k)}\sum_{\mu \in P_{s,n-2,k-1}}
\sum_{i=0}^r
\sum_{{\scriptstyle |\nu|=d'-s} \atop {\scriptstyle l(\nu)=i}}a^{\mu}_{i,\nu}e_{\nu}Q_{T}^{\mu;m}=0,
\label{the:n-23}
\end{equation}
then all constants $a^{\mu}_{r,\nu}$ are equal to $0$.
\end{lemma}
Proof of lemma:
We set
\begin{equation*}
I=\sum_{s=0}^{(k-1)(n-k)}\sum_{\mu \in P_{s,n-2,k-1}}
\sum_{i=0}^r
\sum_{{\scriptstyle |\nu|=d'-s} \atop {\scriptstyle l(\nu)=i}}a^{\mu}_{i,\nu}e_{\nu}Q_{T}^{\mu;m}.
\end{equation*}
From Proposition.\ref{property:n-k} (3), we have $deg_{x_{k+1}}(I)=(k-1)m+r$. 
The leading coefficient of $I$ in $x_{k+1}$ is in
\begin{equation*}
\sum_{s=0}^{(k-1)(n-k)}\sum_{\mu\in P_{s,n-2,k-1}}
\sum_{{\scriptstyle |\nu|=d'-s} \atop {\scriptstyle l(\nu)=r}}a^{\mu}_{r,\nu}e_{\nu}Q_{T}^{\mu;m}.
\end{equation*}

Recall that we have $P_{s,n-2,k-1}=Q_{s,n-2,k-1} \cup P_{s,n-3,k-1}$ and 
this union is disjoint. 
Therefore we can rewrite this as
\begin{eqnarray*}
&&\sum_{s=n-k}^{(k-1)(n-k)}\sum_{\mu\in Q_{s,n-2,k-1}}
\sum_{{\scriptstyle |\nu^{(1)}|=d'-s} \atop {\scriptstyle l(\nu^{(1)})=r}}
a^{\mu}_{r,\nu^{(1)}}e_{\nu^{(1)}}Q_{T}^{\mu;m}  \\
&+&\sum_{s=0}^{(k-1)(n-k-1)}\sum_{\mu\in P_{s,n-3,k-1}}
\sum_{{\scriptstyle |\nu^{(2)}|=d'-s} \atop {\scriptstyle l(\nu^{(2)})=r}}
a^{\mu}_{r,\nu^{(2)}}e_{\nu^{(2)}}Q_{T}^{\mu;m}.
\end{eqnarray*}

We set 
\begin{eqnarray*}
I_1&=&\sum_{s=n-k}^{(k-1)(n-k)}\sum_{\mu\in Q_{s,n-2,k-1}}
\sum_{{\scriptstyle |\nu^{(1)}|=d'-s} \atop {\scriptstyle l(\nu^{(1)})=r}}
a^{\mu}_{r,\nu^{(1)}}e_{\nu^{(1)}}Q_{T}^{\mu;m}   \\
I_2&=&\sum_{s=0}^{(k-1)(n-k-1)}\sum_{\mu\in P_{s,n-3,k-1}}
\sum_{{\scriptstyle |\nu^{(2)}|=d'-s} \atop {\scriptstyle l(\nu^{(2)})=r}}
a^{\mu}_{r,\nu^{(2)}}e_{\nu^{(2)}}Q_{T}^{\mu;m}.
\end{eqnarray*}
First, we show that the constants $a^{\mu}_{r,\nu}$ in $I_1$ are equal to $0$. 

If $r>d'-n+k$, we have $|\mu|<\frac{(k-1)(k-2)}{2}+n-k$. 
On the other hand, if $\mu\in Q_{s,n-2,k-1}$, 
we have $|\mu|\geq \frac{(k-1)(k-2)}{2}+n-k$. Therefore if $r>d'-n+k$, 
the sum in $I_1$ is empty.
We only needs to consider the case when $r\leq d'-n+k$.

We define the following notions.

Let $X=\{ s_1,s_2,\cdots ,s_n \}$ be the set of $n$ positive integers.
For a partition $\nu=(\nu_1,\nu_2,\cdots)$,
we define 
\begin{eqnarray*}
e_{X,i}&=&\sum_{1\leq l_1< \cdots < l_i\leq n}x_{s_{l_1}}\cdots x_{s_{l_i}} \\
e_{X,\nu}&=&\prod_ie_{X,\nu_i} \\
e_{X,i}^{(s_j)}&=&e_i(x_{s_1},\cdots ,x_{s_{j-1}},x_{s_{j+1}},\cdots ,x_{s_n}) \\
e_{X,\nu}^{(s_j)}&=&\prod_{s_i}e_{X,\nu_i}^{(j)}.
\end{eqnarray*}
In particular, if $X=\{ 1,2,\cdots ,n \}$, then we simply write $e_{X,i}^{(j)}$ as $e_{i}^{(j)}$ and
$e_{X,\nu}^{(j)}$ as $e_{\nu}^{(j)}$.

When $r\leq d'-n+k$, 
the terms of the highest degree of $I$ in $x_1$ are in $I_1$. 
For $\mu\in Q_{s,n-2,k-1}$, there exists $\mu'=(\mu'_1,\cdots ,\mu'_{k-2})\in P_{n-3,k-2}$ 
such that $\mu=(n-2,\mu'_1,\cdots ,\mu'_{k-2})$.
In particular, we have $\mu'\in P_{s+k-n,n-3,k-2}$. 
The leading coefficient of $I_1$ in $x_1$ is 
\begin{equation*}
\sum_{s=n-k}^{(k-1)(n-k)}\sum_{\mu'\in P_{s+k-n,n-3,k-2}}
\sum_{{\scriptstyle |\nu^{(1)}|=d'-s} \atop {\scriptstyle l(\nu^{(1)})=r}}
b^{\mu'}_{\nu^{(1)}}e_{\nu^{(1)}-(1^{r})}^{(1)}Q_{T^1}^{\mu';m}
\end{equation*}
where we set $b_{\mu',\nu^{(1)}}
=\frac{(-1)^{(k-1)m+1}m!}{\prod_{s=0}^m(mn+n-1-s)}a^{(n-2,\mu'_1,\cdots)}_{r,\nu^{(1)}}$.
We can rewrite this as
\begin{equation*}
\sum_{s=0}^{(k-2)(n-k)}\sum_{\mu'\in P_{s,n-3,k-2}}
\sum_{{\scriptstyle |\nu^{(1)}|=d'-s+k-n} \atop {\scriptstyle l(\nu^{(1)})=r}}
b^{\mu'}_{\nu^{(1)}}e_{\nu^{(1)}-(1^{r})}^{(1)}Q_{T^1}^{\mu';m}.
\end{equation*}

Since $e_{\nu^{(1)}-(1^{r})}^{(1)}=e_{mem(T^1), \nu^{(1)}-(1^{r})}$, this is rewritten as
\begin{equation*}
\sum_{s=0}^{(k-2)(n-k)}\sum_{\mu'\in P_{s,n-3,k-2}}
\sum_{{\scriptstyle |\nu^{(1)}|=d'-s+k-n} \atop {\scriptstyle l(\nu^{(1)})=r}}
b^{\mu'}_{\nu^{(1)}}e_{mem(T^1), \nu^{(1)}-(1^{r})}Q_{T^1}^{\mu';m}.
\end{equation*}
The shape of $T^1$ is $(n-k+1,1^{k-2})$. Thus $T^1$ has $n-1$ boxes.
By assumption on $n$, all $b^{\mu'}_{\nu^{(1)}}$ are equal to $0$.
Thus we have $a^{(n-2,\mu'_1,\cdots)}_{r,\nu^{(1)}}=0$, hence we get $I_1=0$.

We next consider $I_2$. 
The leading coefficient of $I_2$ in $x_{k+1}$ is
\begin{equation}
\sum_{s=0}^{(k-1)(n-k-1)}\sum_{\mu\in P_{s,n-3,k-1}}
\sum_{{\scriptstyle |\nu^{(2)}|=d'-s} \atop {\scriptstyle l(\nu^{(2)})=r}}
c^{\mu}_{\nu^{(2)}}e_{\nu^{(2)}-(1^{r})}^{(k+1)}
Q_{T^{k+1}}^{\mu;m}
\label{eq:98765}
\end{equation}
where we set $c^{\mu}_{\nu^{(2)}}=(-1)^{(k-2)m}a^{\mu}_{r,\nu^{(2)}}$.

Since $e_{\nu^{(2)}-(1^{r})}^{(k+1)}=e_{mem(T^{k+1}), \nu^{(2)}-(1^{r})}$, we can rewrite (\ref{eq:98765}) as
\begin{equation*}
\sum_{s=0}^{(k-1)(n-k-1)}\sum_{\mu\in P_{s,n-3,k-1}}
\sum_{{\scriptstyle |\nu^{(2)}|=d'-s} \atop {\scriptstyle l(\nu^{(2)})=r}}
c^{\mu}_{\nu^{(2)}}e_{mem(T^{k+1}), \nu^{(2)} -(1^{r})}
Q_{T^{k+1}}^{\mu;m}.
\end{equation*}
The tableau $T^{k+1}$ has $n-1$ boxes.
By assumption on $n$, all $c^{\mu}_{\nu^{(2)}}$ are equal to $0$. 
Thus, all $a^{\mu}_{r,\nu}$ are equal to $0$.

Thus lemma follows. Therefore proposition also follows. $\Box$ }
\end{proof}

From Theorem.\ref{basis:n-k+1} and Proposition.\ref{proposition:n-k+1}, we obtain the following corollary.
\begin{corollary}
Let $T\in ST(\eta(n,k))$.
$\gamma_{T}({\bf QI_m})$ is free module over $\Lambda_n$ and 
$\bigl\{ Q_{T}^{\mu;m} \bigr\}_{\mu\in P(n-2;k-1)}$ 
is a free basis. 
\label{cor:n-k333}
\end{corollary}

\begin{proof}
{\rm
In this proof, we simply write $Q_{T}^{\mu;m}$ as $Q^{\mu}$. 
Using Proposition.\ref{proposition:n-k+1}, $\bigl\{ Q^{\mu} \bigr\}$ are free over $\Lambda_n$. 

Since $\displaystyle{H(\gamma_T({\bf QI_m}^*); t)=
\sum_{s=0}^{(k-1)(n-k)}t^{(k-1)nm+\frac{k(k-1)}{2}}\sum_{s=0}^{(k-1)(n-k)}
p_{s,n-2,k-1} t^s}$, 
we have 
\begin{equation*}
\gamma_{T}({\bf QI_m})=\bigoplus_{d\geq (k-1)nm+\frac{k(k-1)}{2}}\gamma_{T}({\bf QI_m})[d]. 
\end{equation*}
Let $d$ be a non-negative integer such that $d\geq (k-1)nm+\frac{k(k-1)}{2}$. 
We show that the subspace of $\gamma_{T}({\bf QI_m})[d]$ is generated by 
$\bigl\{ Q^{\mu} \bigr\}$ over $\Lambda_n$. We show this by induction on $d$.

When $d=(k-1)nm+\frac{k(k-1)}{2}$, the coefficient of $t^{(k-1)nm+\frac{k(k-1)}{2}}$ 
in $H(\gamma_T({\bf QI_m}^*); t)$ is equal to $1$.
Therefore, $\gamma_{T}({\bf QI_m})[d]$ is a space spanned by $Q^{(k-2,k-1,\cdots ,0)}$. 
Thus the statement follows in the case $d=(k-1)nm+\frac{k(k-1)}{2}$. 

When $d\geq (k-1)nm+\frac{k(k-1)}{2}+1$, we assume that the statement holds with all numbers less than $d$. 
We denote by $V$ the vector space over $\mathbb{Q}$ spanned by $\{ Q^{\mu} \}_{\mu\in P(n-2;k-1)}$.

Take $f\in \gamma_{T}({\bf QI_m})[d]$. Since Theorem.\ref{basis:thn-1},
we can find $g\in V[d]$ such that $[f]=[g]$ in $\gamma_{T}({\bf QI_m}^*)$. 
Thus, we have $f-g\in \langle e_1,\cdots,e_n\rangle $. 
This is expressed as 
\begin{equation*}
f-g=\sum_{s\geq 1}A_su_s
\end{equation*}
where $A_s\in \Lambda_n^s$ and $u_s\in \gamma_{T}(\bf {QI_m})$. 

Since $\gamma_{T}(\bf {QI_m})$ is a graded space,
we can decompose $u_s=\sum_{i\geq 0}u_{s,i}$ where $u_{s,i}\in \gamma_{T}({\bf {QI_m}})[i]$.
We have $deg(A_su_{s,i})=s+i$. 
Thus, we have 
\begin{equation*}
f-g=\sum_{l\geq 0}\sum_{s+i=l}A_s u_{s,i}.
\end{equation*}
Since $f-g\in \gamma_{T}({\bf QI_m})[d]$, 
we get $\displaystyle{\sum_{l\neq d}\sum_{s+i=l}A_s u_{s,i}=0}$. 
Therefore, we have 
\begin{equation*}
f-g=\sum_{s\geq 1}A_s u_{s,d-s}.
\end{equation*}

$A_s$ has degree at least $1$, therefore $u_{s,d-s}$ has the degree less than $d$. 
By induction assumption, $u_{s,d-s}$ can be expressed as 
\begin{equation*}
u_{s,d-s}=\sum_{l}B_l v_l
\end{equation*}
where $B_l\in \Lambda_n$ and $v_l\in V$.
Thus, the statement follows. $\Box$
}
\end{proof}

\section{The operator $L_m$}
The operator $L_m$ is defined as

\begin{equation*}
L_m=\sum_{i=1}^n\frac{\partial^2}{\partial x_i^2}-2m\sum_{1\leq i < j\leq n}
\frac{1}{x_i-x_j}\bigl( \frac{\partial}{\partial x_i} - \frac{\partial}{\partial x_j} \bigr)
\end{equation*}

This operator is discussed in \cite{etingof} and \cite{feigin}. It is related the quasiinvariants.
In \cite{feigin} Feigin and Veselov showed that the operator $L_m$ preserves ${\bf QI_m}$.
We consider how $L_m$ acts on our polynomial $Q_{T}^{\mu;m}$.
In \cite{bandlow}, for $T(1,2)$  
Bandlow and Musiker showed that the following formulas for the action of $L_m$.

\begin{theorem}[\cite{bandlow}]
Let $k,m$ be non-negative integers. 

Then we have $L_m(Q_{T(1,2)}^{k;m})=k(k-1)Q_{T(1,2)}^{k-2;m}$ for $k\geq 2$ 
and $L_m(Q_{T(1,2)}^{k;m})=0$ for $k=0,1$.
\end{theorem}

We extend this formulas.
We set $T=T(1,2,\cdots ,k)$.
To write formulas simply, we define the following polynomials.

\begin{definition}
{\rm 
Let $\alpha =(\alpha_1,\alpha_2,\cdots ,\alpha_{k-1})\in \mathbb{Z}^{k-1}$. 

We define a polynomial $Q_{T}^{\alpha;m}$ as follows:
\begin{equation}
Q_{T}^{\alpha;m}=
\begin{vmatrix}
R_{T;1,2}^{\alpha_1;m}     & R_{D;1,2}^{\alpha_2;m}     & \cdots & R_{T;1,2}^{\alpha_{k-1};m} \\
R_{T;2,3}^{\alpha_1;m}     & R_{T;2,3}^{\alpha_2;m}     & \cdots & R_{T;2,3}^{\alpha_{k-1};m} \\
  \vdots               &         \vdots         & \ddots &        \vdots          \\
R_{T;k-1,k}^{\alpha_1;m}   & R_{T;k-1,k}^{\alpha_2;m}   & \cdots & R_{T;k-1,k}^{\alpha_{k-1};m}
\end{vmatrix} 
\label{basis:n-k}
\end{equation}
when $\alpha_i\geq 0$ $i=1,\cdots ,k-1$. Otherwise we define $Q_{T}^{\alpha;m}=0$.}
\end{definition}
\begin{remark}
{\rm If $\alpha$ is a partition, 
$Q_{T}^{\alpha;m}$ is equal to a polynomial defined in Definition.\ref{basis:n-k}.
If $\alpha\in \mathbb{Z}_{\geq 0}^{k-1}$, 
$Q_T^{\alpha;m}$ is equal to $Q_T^{\mu;m}$ up to a sign where $\mu$ is a partition sorted $\alpha$.
}
\end{remark}

We obtain  the following formulas for the action of $L_m$.
To write the formula simply, 
for $\alpha =(\alpha_1,\alpha_2,\cdots ,\alpha_{k-1})\in \mathbb{Z}^{k-1}$ we define 
\begin{equation*}
\alpha^{(i,j)}=(\alpha_1 ,\cdots ,\alpha_{i-1},\alpha_i -1,\alpha_{i+1},
\cdots ,\alpha_{j-1}, \alpha_j -1,\alpha_{j+1},\cdots,\alpha_n).
\end{equation*}

\begin{theorem}
Let $\alpha =(\alpha_1,\cdots ,\alpha_{k-1}) \in \mathbb{Z}^{k-1}$ and take $T\in ST(\eta(n,k))$.
Then we have 
\begin{eqnarray*}
L_m(Q_T^{\alpha;m})
&=&\sum_{i=1}^n\alpha_i (\alpha_i -1)
Q_T^{(\alpha_1 ,\cdots ,\alpha_i -2,\cdots ,\alpha_n);m} 
+2m\sum_{1\leq i < j\leq n} \bigl(-\alpha_i Q_T^{\alpha^{(i,j)};m} \\
&+&\sum_{{\scriptstyle \alpha_i -2\geq s>t \geq 0} \atop {\scriptstyle s+t=\alpha_i+\alpha_j-2}}
(s-t)Q_T^{(\alpha_1 ,\cdots ,\alpha_{i-1},s,\alpha_{i+1},
\cdots ,\alpha_{j-1},t,\alpha_{j+1},\cdots,\alpha_n);m} \bigr).
\end{eqnarray*}
\label{thmLm}
\end{theorem}

This follows from following lemma.
We define a polynomial $R_{T;1,2,3}^{s,t;m}$ as
\begin{equation*}
R_{T;1,2,3}^{s,t;m}=
\begin{vmatrix}
R_{T;1,2}^{s;m}     & R_{T;1,2}^{t;m}   \\
R_{T;2,3}^{s;m}     & R_{T;2,3}^{t;m}  
\end{vmatrix} 
.
\end{equation*}

\begin{lemma}
{\rm(1)} we have
\begin{equation*}
L_m(fg)=L_m(f)g+fL_m(g)+
2\sum_{i=1}^n\bigl( \frac{\partial}{\partial x_i}f \bigr) \bigl( \frac{\partial}{\partial x_i}g\bigr). 
\end{equation*} \\
{\rm(2)}
Let $k$ be a non-negative integer and $m$ be a non-negative integer. Then we have 
\begin{equation*}
k\int_{x_i}^{x_j}t^{k-1}\prod_{s=1}^n(t-x_s)^m
=-m\sum_{r=1}^n\int_{x_i}^{x_j}t^k(t-x_r)^{m-1}\prod_{s\neq r}(t-x_s)^m dt.
\end{equation*} \\
{\rm(3)}
Let $k,l$ be non-negative integers such that $k>l$. Then we have
\begin{align}
&\sum_{i=1}^n\bigl( \frac{\partial}{\partial x_i}R_{T;1,2}^{k;m} \bigr) 
\bigl( \frac{\partial}{\partial x_i}R_{T;1,3}^{l;m}\bigr)
-\bigl( \frac{\partial}{\partial x_i}R_{T;1,3}^{k;m} \bigr) 
\bigl( \frac{\partial}{\partial x_i}R_{T;1,2}^{l;m}\bigr) \\
&=m \bigl(-l R_{T;1,2,3}^{k-1,l-1;m}
+\sum_{{\scriptstyle k-2\geq s>t \geq 0} \atop {\scriptstyle s+t=k+l-2}} 
(s-t)R_{T;1,2,3}^{s,t;m} \bigr). 
\label{eq-Lm}
\end{align}
\end{lemma}

\begin{proof}
{\rm
(1)
It follows from Leibniz's rule. \\
(2)
It follows from following identity:
\begin{equation*}
\int_{x_i}^{x_j}\frac{\partial}{\partial t}t^{k}\prod_{s=1}^n(t-x_s)^m=0.
\end{equation*} \\
(3)
When $m=0$, it follows from $R_{T;1,2}^{k;m}=\frac{x_2^{k+1}-x_1^{k+1}}{k+1}$.
We consider the case $m\geq 1$.
 
We show this formula by induction on $k-l$. 
We define $f(t,x)=\prod_{s=1}^n(t-x_s)^m$ and
$f_i(t,x)=(t-x_i)^{m-1}\prod_{s\neq i}(t-x_s)^m$.
When $k-l=1$, l.h.s. of (\ref{eq-Lm}) is equal to
\begin{align*}
&m^2\sum_{i=1}^n\int_{x_1}^{x_2}t^kf_i(t,x) dt
\int_{x_1}^{x_3}u^{k-1}f_i(u,x) du\\
-&m^2\sum_{i=1}^n\int_{x_1}^{x_3}t^kf_i(t,x) dt
\int_{x_1}^{x_2}u^{k-1}f_i(u,x) du.
\end{align*}

So this is equal to
\begin{align*}
&m^2\sum_{i=1}^n\int_{x_1}^{x_2}t^{k-1} \{(t-x_i)+x_i \} f_i(t,x) dt
\int_{x_1}^{x_3}u^{k-1}f_i(u,x) du\\
-&m^2\sum_{i=1}^n\int_{x_1}^{x_3}t^{k-1} \{(t-x_i)+x_i \} f_i(t,x) dt
\int_{x_1}^{x_2}u^{k-1}f_i(u,x) du\\
=&m^2\sum_{i=1}^n\int_{x_1}^{x_2}t^{k-1} f(t,x) dt
\int_{x_1}^{x_3}u^{k-1} f_i(u,x) du\\
-&m^2\sum_{i=1}^n\int_{x_1}^{x_3}t^{k-1} f(t,x) dt
\int_{x_1}^{x_2}u^{k-1} f_i(u,x) du.
\end{align*}
Using (2), we have
\begin{equation*}
l.h.s.\ of \ (\ref{eq-Lm})
=-m(k-1)R_{T;1,2,3}^{k-1,k-2;m}.
\end{equation*}

We consider the case $k-l=2$. Calculating it in the same way, we have 
\begin{eqnarray*}
l.h.s.\ of\ (\ref{eq-Lm}) &=&-m(k-2)R_{T;1,2,3}^{k-1,k-3;m}\\
&&+m^2\sum_{i=1}^n\int_{x_1}^{x_2}t^{k-1} f_i(t,x) dt
\int_{x_1}^{x_3}x_iu^{k-2} f_i(u,x) du\\
&&-m^2\sum_{i=1}^n\int_{x_1}^{x_3}t^{k-1} f_i(t,x) dt
\int_{x_1}^{x_2}x_iu^{k-2} f_i(u,x) du.
\end{eqnarray*}
From $x_i=u-(u-x_i)$, we get
\begin{eqnarray*}
l.h.s.\ of\ (\ref{eq-Lm})&=&-m(k-2)R_{T;1,2,3}^{k-1,k-3;m}\\
&&+m^2\sum_{i=1}^n\int_{x_1}^{x_2}t^{k-1} f_i(t,x) dt
\int_{x_1}^{x_3}\{ u-(u-x_i) \} u^{k-2} f_i(u,x) du\\
&&-m^2\sum_{i=1}^n\int_{x_1}^{x_3}t^{k-1} f_i(t,x) dt
\int_{x_1}^{x_2} \{ u-(u-x_i) \} u^{k-2} f_i(u,x) du.
\end{eqnarray*}
It is equal to $-m(k-2)R_{T;1,2,3}^{k-1,k-3;m}$. 
Thus theorem holds when $k-l=2$.

When $k-l\geq 3$, we assume that 
the formula (\ref{eq-Lm}) holds with all numbers less than $k-l$.
Calculating l.h.s. of (\ref{eq-Lm}) in the same way, we have 
\begin{align*}
&l.h.s.\ of\ (\ref{eq-Lm}) \\
=&-mlR_{T;1,2,3}^{k-1,l-1;m}+m(k-1)R_{T;1,2,3}^{k-2,l;m}\\
+&\sum_{i=1}^n\bigl( \frac{\partial}{\partial x_i}R_{T;1,2}^{k-1;m} \bigr) 
\bigl( \frac{\partial}{\partial x_i}R_{T;1,3}^{l+1;m}\bigr)
-\bigl( \frac{\partial}{\partial x_i}R_{T;1,3}^{k-1;m} \bigr)
\bigl( \frac{\partial}{\partial x_i}R_{T;1,2}^{l+1;m}\bigr)
\end{align*}
So theorem follows by induction assumption. $\Box$
}
\end{proof}

\begin{acknowledgments}
{\rm
I would like to thank Professor Etsuro Date for introducing me to
this subject of the quasiinvariants and for his many valuable advices. 
I would also like to thank Professor Misha Feigin for his interest 
and encouragement during preparation of this paper 
and for useful comments while he was fully occupied.
}
\end{acknowledgments}

{\rm
\begin{description}
\item{Address}:
Graduate School of Information Science and Technology, 
Osaka University, Toyonaka, Osaka 560-0043, Japan

\item{E-mail}:t-tsuchida@ist.osaka-u.ac.jp
\end{description}
}

\end{document}